\documentclass[11pt,a4paper]{amsart}
\usepackage{amsmath,amsfonts,amsxtra,amssymb,latexsym,amsthm,upref}
\usepackage[mathscr]{eucal}
\usepackage[latin1]{inputenc}
\usepackage{upref}
\usepackage{enumerate}
\usepackage{nicefrac}
\usepackage{verbatim}
\usepackage{a4wide}

\makeatletter
\@namedef{subjclassname@2010}{%
 \textup{2010} Mathematics Subject Classification}
\makeatother

\theoremstyle{plain}
\newtheorem{theorem}{Theorem}
\newtheorem{corollary}{Corollary}
\newtheorem{proposition}{Proposition}
\newtheorem{lemma}{Lemma}
\newtheorem*{ethm}{Theorem of Edwards (1957)}
\newtheorem*{clemma}{Composition Lemma}
\newtheorem*{rlemma}{Representation Lemma}
\newtheorem*{dlemma}{Density Lemma}
\newtheorem*{elemma}{Extension Lemma}
\newtheorem*{slemma}{Separation Lemma}

\theoremstyle{definition}

\newtheorem{remark}{Remark}
\newtheorem*{example*}{Example}

\theoremstyle{remark}
\newtheorem{step}{Step}

 \newcommand{\ts}{\textstyle}

 \newcommand{\scra}{\mathscr{A}}

 \newcommand{\scrg}{\mathscr{G}}
 \newcommand{\scrm}{\mathscr{M}}
 \newcommand{\scrw}{\mathscr{W}}

 \newcommand{\caln}{\mathscr{N}}
 \newcommand{\calp}{\mathscr{P}}

 \newcommand{\setC}{\mathbb{C}}
 
 \newcommand{\setH}{\mathbb{H}}
 \newcommand{\setN}{\mathbb{N}}
 \newcommand{\setP}{\mathbb{P}}
 \newcommand{\setQ}{\mathbb{Q}}
 \newcommand{\setR}{\mathbb{R}}
 \newcommand{\setU}{\mathbb{U}}
 \newcommand{\setZ}{\mathbb{Z}}

 \newcommand{\ve}{\varepsilon}
 \newcommand{\vp}{\varphi}
 \newcommand{\vr}{\varrho}
 \newcommand{\vt}{{\vartheta}}
 \newcommand{\om}{\omega}
 \newcommand{\eps}{\epsilon}
 
 \newcommand{\Beta}{\mathrm{B}}

 \newcommand{\ol}{\overline}
 
 \newcommand{\wh}{\widehat}
 \newcommand{\wt}{\widetilde}
 \newcommand{\wrt}{\widetilde{\phantom{o}}}
 \newcommand{\pa}{\partial}
 \newcommand{\DOT}{\text{\rm\larger[4]{.}}}
 \newcommand{\CDOT}{{\raisebox{.4ex}{\text{\rm\larger[3]{.}}}}}
 \newcommand{\AST}{\text{\raisebox{-.5ex}{\huge$\ast$}}}

 \DeclareMathOperator{\kleino}{o}

 \DeclareMathOperator{\Span}{span}
 \DeclareMathOperator{\re}{Re}
 \DeclareMathOperator{\conv}{conv}
 \DeclareMathOperator{\cone}{cone}
 \DeclareMathOperator{\algint}{algint}
 \DeclareMathOperator{\aff}{aff}
 \DeclareMathOperator{\sgn}{sgn}

 \newenvironment{enumeratea}{\begin{enumerate}%
    [\hspace*{1.95em}\upshape a)]}
    {\end{enumerate}}
 \newcommand{\itemref}[1]{\ref{#1})}

\begin{document}

\baselineskip=14pt

\title[Weighted inversion of general Dirichlet series]
{Weighted inversion of general Dirichlet series}

\author[Helge Gl\"{o}ckner]{Helge Gl\"{o}ckner$^*$}

\thanks{$^*$The first author was supported by Deutsche Forschungsgemeinschaft,
GZ: GL 357/5--2.}

\address{Universität Paderborn\\
Institut f\"{u}r Mathematik\\
Warburger Str. 100\\
D-33098 Paderborn\\
Germany}

\email{glockner@math.upb.de}

\author{Lutz G. Lucht}
\address{Technische Universität Clausthal \\
Institut f\"{u}r Mathematik \\
Erzstr. 1 \\
D-38678 Clausthal-Zellerfeld \\
Germany}

\email{lg.lucht@cintech.de}

\date{August 1, 2012}

\begin{abstract}
Inversion theorems of Wiener type are essential tools in analysis and number
theory. We derive a weighted version of an inversion theorem of Wiener type
for general Dirichlet series from that of Edwards from 1957, and we outline
an alternative proof based on the duality theory of convex cones and
extension techniques for characters of semigroups. Variants and arithmetical
applications are described, including the case of multidimensional weighted
generalized Dirichlet series.
\end{abstract}



\maketitle

%%%%%%%%%%%%%%%%%%%%%%%%%%%%%%%%%%%%%%%%%%%%%%%%%%%%%%%%%%%%%%%%%%%%%%%%%%%%%

\section{Introduction} \label{S1}

%%%%%%%%%%%%%%%%%%%%%%%%%%%%%%%%%%%%%%%%%%%%%%%%%%%%%%%%%%%%%%%%%%%%%%%%%%%%%

\noindent
By
$\scra:=\scra(\Lambda):=\{a\colon\Lambda\to\setC\text{ with }\|a\|<\infty\}$
we denote the class of complex-valued functions defined on an
additive semigroup $\Lambda\subseteq [0,\infty)$ with $0\in\Lambda$ and
satisfying
\begin{equation*}
\|a\|:=\sum_{\lambda\in\Lambda}\,|a(\lambda)|<\infty.
\end{equation*}
Then, under the usual linear operations and the convolution defined by
\begin{equation} \label{eq1}
c(\lambda) := (a*b)(\lambda) :=
\sum_{\substack{\lambda',\lambda''\in\Lambda \\ \lambda'+\lambda''=\lambda}}
a(\lambda')\,b(\lambda'') \qquad (\lambda\in\Lambda),
\end{equation}
$\scra$ forms a commutative unitary Banach algebra. In fact, the convolution
is well defined by~\eqref{eq1} because of $\|c\|\leq\|a\|\cdot\|b\|<\infty$
for $a,b\in\scra$ with the above norm $\|\DOT\|$, and the function
$\ve$ given by $\ve(\lambda)=\delta_{0,\lambda}$ with Kronecker's symbol
$\delta\colon\Lambda^2\to\{0,1\}$ serves as unity (cf.~\cite{HZ1956}).

Let $\setH:=\{s\in\setC:\re{s}>0\}$ denote the open right half plane.
Then the Banach algebra $\scra$ is isomorphic to the Banach algebra
$\wt{\scra}=\wt{\scra}(\Lambda)$ of absolutely convergent \emph{general
Dirichlet series}
\begin{equation*}
\wt{a}(s) = \sum_{\lambda\in\Lambda}\,a(\lambda)\,e^{-\lambda s}
            \qquad (s\in\ol{\setH})
\end{equation*}
associated with $a\in\scra$, under the usual linear operations, pointwise
multiplication and norm $\|\wt{a}\|:=\|a\|$. This allows to switch between
complex sequences and Dirichlet series.

For instance, inversion in $\wt{\scra}$ is equivalent to that in $\scra$, and
the inversion theorem of Edwards~\cite{Ed1957} may be formulated as

\begin{ethm}
The multiplicative group of $\scra$ is
$\scra^*=\{a\in\scra:0\notin\ol{\wt{a}(\setH)}\}$.
\end{ethm}

\begin{remark} \label{R1}
Special cases of Edwards' theorem trace back to Wiener \cite{Wi1932} for
Fourier and power series and to Hewitt and Williamson \cite{HW1957} for
ordinary Dirichlet series. Elementary proofs of the latter were given by
Spilker and Schwarz \cite{SS1979}, Goodman and Newman \cite{GN1984}.
\end{remark}

We aim for investigating the convergence quality of $\wt{a}$ on
$\ol{\setH}$ and refine the above norm with the help of \emph{weight
functions} on $\Lambda$, i.e. functions $w\colon\Lambda\to(0,\infty)$
satisfying $w(0)=1$ and $w(\lambda'+\lambda'')\leq w(\lambda')\,w(\lambda'')$
for all $\lambda',\lambda''\in\Lambda$. We let $\scrw=\scrw(\Lambda)$ be the
set of all weight functions~$w$ on $\Lambda$ which are \emph{admissible} in
the sense that

\begin{enumeratea}
\item \label{ita}
$w(\lambda)\geq 1$ for all $\lambda\in\Lambda$, and
\item \label{itb}
$\lim_{k\to\infty}{\!\textstyle\sqrt[k]{w(k\lambda)}}=1~$ for all
$\lambda\in\Lambda$.%
\footnote{The sum $\lambda+\cdots+\lambda$ with $k\in\setN$ summands
$\lambda\in\Lambda$ is written as $k\lambda$.}
\end{enumeratea}

Note that $\scrw$ is closed under pointwise multiplication. If $w\in\scrw$,
then the functions $a\in\scra$ with the weighted norm
\begin{equation*}
\|a\|_w:=\sum_{\lambda\in\Lambda}\,|a(\lambda)|\,w(\lambda)<\infty
\end{equation*}
form a commutative unitary Banach algebra
$\scra_w=\scra_w(\Lambda)\subseteq\scra=\scra_1$
(see Section~\ref{algebras} for a more detailed discussion of weighted
semigroup algebras, even for arbitrary weights).

Since weight functions are submultiplicative, also $\|\DOT\|_w$ is
submultiplicative. The conditions \ref{ita}) and \ref{itb}) ensure that the
spectrum of $\scra_w$ can be identified with the set of all bounded
characters of~$\Lambda$ (see Section~\ref{S4}). They also restrict the growth
of admissible weight functions~$w$ such that in relevant cases, e.g. for
non-decreasing weight functions, we obtain
\begin{equation} \label{bigo}
w(\lambda)\ll e^{\vt\lambda}\quad\text{for every $\vt>0$}
\end{equation}
(see Remark \ref{R3} below).%
\footnote{For $f\colon\Lambda\to\setC$ and $g\colon\Lambda\to\setR_+$
we write $f\ll g$ if
$\sup_{\lambda\in\Lambda}\frac{|f(\lambda)|}{g(\lambda)}<\infty$
(\emph{Vinogradov $\ll$ symbol}).}
Whenever \eqref{bigo} holds, the abscissa of absolute convergence of $\wt{a}$
equals that of $(aw)\wrt$ for any $a\colon\Lambda\to\setC$. Moreover, for
non-decreasing $w\in\scrw$, the remainder term estimate
\begin{equation*}
 \sum_{\substack{\lambda\in\Lambda \\ \lambda\geq\ell}}\,
 |a(\lambda)|\,e^{-\lambda s} = \kleino\Big(\frac{1}{w(\ell)}\Big)
 \qquad (\ell\in\Lambda,\,\ell\to\infty)
\end{equation*}
holds uniformly for $s\in\ol{\setH}$.%
\footnote{For $f\colon\Lambda\to\setC$ and $g\colon\Lambda\to\setR_+$ we
write $f(\lambda)=\kleino(g(\lambda))$ for $\lambda\to\infty$ if
$\lim_{\lambda\to\infty}\frac{f(\lambda)}{g(\lambda)}=0$
(\emph{Landau $\kleino$ symbol}).}%
\footnote{The series is dominated by
$\frac{1}{w(\ell)}\,\sum_{\lambda\geq\ell} |a(\lambda)|\,w(\lambda)$.}
For instance, $w(\lambda)=(\lambda+1)^c$ defines an admissible ascending
weight function for any $c\geq 0$. Note that for $a\in\scra=\scra_1$ the
derivative
\begin{equation*}
 \wt{a}{\,'}(s)=-\sum_{\lambda\in \Lambda}\,
 \lambda\,a(\lambda)\,e^{-\lambda s}
\end{equation*}
of $\wt{a}$ converges absolutely only for $s\in\setH$, in general. But for 
$w(\lambda)=(1+\lambda)^k$, $k\in\setN$, the Dirichlet series $\wt{a}$ of 
$a\in\scra_w$ \emph{and its derivatives up to order $k$} converge absolutely 
for $s\in\ol{\setH}$, and the quality of convergence is described by
\begin{equation*}
\sum_{\substack{\lambda\in\Lambda \\ \lambda\geq\ell}}\,
\lambda^j\,|a(\lambda)|\,e^{-\lambda s} = \kleino\Big((\ell+1)^{j-k}\Big)
\qquad (\ell\in\Lambda,\,\ell\to\infty),
\end{equation*}
uniformly for $s\in\ol{\setH}$ and $j=0,\ldots,k$.

\begin{remark} \label{R2}
Condition \ref{itb}) is equivalent to
$\inf{\big\{\sqrt[k]{w(k\lambda)}:k\in\setN\big\}}=1$.

Obviously \ref{itb}) yields the infimum condition. Conversely, the infimum
condition implies that for every $\eps>0$ there exists $k_0$ such that
$\sqrt[k_0]{w(k_0\lambda)}<1+\eps$. Since every $k\geq k_0$ can be
written as $k=n_k k_0+r_k$ with $0\leq r_k<k_0$, we have
\begin{equation*}
\sqrt[k]{w(k\lambda)} \leq
\sqrt[k]{w(r_k\lambda)}\sqrt[k]{w(k_0\lambda)^{n_k}}.
\end{equation*}
Here the first term on the right tends to $1$ as $k\to\infty$ (as there are
only finitely many choices for $r_k$), and the second term can be estimated
via
\begin{equation*}
\sqrt[k]{w(k_0\lambda)^{n_k}} < \sqrt[k]{(1+\eps)^{k_0n_k}} =
(1+\eps)^{k_0 n_k/k}\leq 1+\eps.
\end{equation*}
\end{remark}

\begin{remark} \label{R3}
If $w\in\scrw(\Lambda)$ is \emph{slowly decreasing},%
\footnote{A function $f\colon\Lambda\to[0,\infty)$ is called \emph{slowly
decreasing}
if $\liminf{(f(\lambda)-f(\lambda'))}\geq 0$ whenever
$1\leq\lambda/\lambda'\to 1$ and $\lambda\to\infty$.}
then $w(\lambda)\ll e^{\vt\lambda}$ for every $\vt>0$.

To see this, fix some $\lambda_1\in\Lambda\setminus\{0\}$. Then, for $\vt>0$
and $\lambda\in\Lambda$ sufficiently large, there exists $k\in\setN$ such
that $(k-1)\lambda_1<\lambda\leq k\lambda_1$ and
\begin{equation*}
w(\lambda) = w(k\lambda_1)-(w(k\lambda_1)-w(\lambda))
           \leq w(k\lambda_1)+1 \ll e^{\vt\lambda_1 k}+1 \ll e^{\vt\lambda}.
\end{equation*}
\end{remark}

\begin{remark} \label{R4}
In general, conditions \ref{ita}) and \ref{itb}) do \emph{not} imply that
$w(\lambda)\ll e^{\vt\lambda}$ for any $\vt>0$.%
\footnote{This was erroneously stated in \cite{Lu2008}.}
\end{remark}

Prototypes of general Dirichlet series are power series for $\Lambda=\setN_0$,
and ordinary Dirichlet series for $\Lambda=\log{\setN}$.
Both of these additive semigroups admit \emph{free generating sets},%
\footnote{A subset $\emptyset\neq\Beta\subseteq\Lambda$ is called a
\emph{free generating set} of an additive semigroup
$\Lambda\subseteq [0,\infty)$, if every $\lambda\in\Lambda$ has a unique
representation of the form  $\lambda=\sum_{\beta\in \Beta} \nu_\beta\beta$
with $\nu_\beta\in\setN_0$ and $\nu_\beta\not=0$ for only finitely many
$\beta\in\Beta$. Then $\Beta$ is $\setQ$-linearly independent in $\setR$.}
namely $\{1\}$ and $\log{\setP}$, respectively, where
$\setP=\{2,3,5,\ldots\}$ denotes the set of primes.

The purpose of this note is to outline two different proofs of the following

\begin{theorem} \label{T1}
Let $\Lambda\subseteq[0,\infty)$ be an additive semigroup with $0\in\Lambda$.
Then, for $w\in\scrw$, the Banach algebra $\scra_w=\scra_w(\Lambda)$ has the
multiplicative group
\begin{equation*}
 \scra_w^*=\scra_w^*(\Lambda)=\{a\in\scra_w: 0\notin\ol{\wt{a}(\setH)}\}.
\end{equation*}
\end{theorem}

Theorem \ref{T1} is a weighted inversion theorem of Wiener type for general
Dirichlet series. The necessity of the inversion condition is immediate, for
if $1/\wt{a}(s)$ can be represented by an absolutely convergent Dirichlet
series then it is bounded on $\setH$ and hence $\wt{a}(s)$ must be bounded
away from zero on $\setH$. Note that the inversion condition does not depend
on the weight function $w\in\scrw$.

In Section \ref{S4} we derive the proof of Theorem~\ref{T1} from the Theorem
of Edwards and provide an alternative proof based on an approximation of
multiplicative linear functionals on $\scra_w$ (taking $w=1$, this also gives
a new proof for the Theorem of Edwards). The required Density Lemma is
established step by step in Sections \ref{S5} to \ref{S7}. For discrete
additive semigroups $\Lambda\subseteq [0,\infty)$, the above Theorem \ref{T1}
and the strategy of its proof are described in a preliminary (incomplete)
paper of Lucht and Reifenrath \cite{LR1996}. The Lévy generalization and a
multidimensional Lévy version of Theorem \ref{T1}, Theorems \ref{T3} and
\ref{T5}, are established in Sections \ref{S8}, which also deals with an
arithmetical application (Proposition \ref{P3}).

For further recent studies of weighted convolution algebras of subsemigroups
of $\setR$ (with a different thrust), the reader is referred to \cite{DD2009}.
In the case of groups (rather than semigroups), the admissibility
condition b) is related to symmetry properties of weighted group algebras,
and has been attributed to Gelfand, Naimark, Raikov and \v{S}hilov in some
recent works (see \cite[p.\,796]{FG2003}, \cite[Definition 1.2(a)]{FG2006},
\cite[Definition 3.1(a)]{KM2012}).

%%%%%%%%%%%%%%%%%%%%%%%%%%%%%%%%%%%%%%%%%%%%%%%%%%%%%%%%%%%%%%%%%%%%%%%%%%%%%

\section{Tools from Gelfand's theory} \label{S2}

%%%%%%%%%%%%%%%%%%%%%%%%%%%%%%%%%%%%%%%%%%%%%%%%%%%%%%%%%%%%%%%%%%%%%%%%%%%%%

\noindent
With any commutative Banach algebra $A$ Gelfand's theory associates the
space $\Delta(A)$ of homomorphisms of $A$ to the complex field, i.e. the
nontrivial multiplicative linear functionals $h\colon A\to\setC$. The
supremum norm on $\Delta(A)$ is related to the norm on $A$. This helps to
characterize the invertible elements of $A$ (cf., for instance,
\cite[Theorems 18.3 and 18.17]{Ru1970}):

\begin{theorem} \label{T2}
In a commutative Banach algebra with unity $u$ the following assertions
hold.
\begin{enumeratea}
\item \label{2ita}
If $a\in A$ satisfies $\|a-u\|<1$, then $a$ is invertible in $A$.
\item \label{2itb}
If $a\in A$ and $h\in\Delta(A)$, then $|h(a)|\leq\|a\|$.
\item \label{2itc}
If $a\in A$ is invertible in $A$, then $h(a)\neq 0$ for all $h\in\Delta(A)$,
and vice versa.
\end{enumeratea}
\end{theorem}

From \itemref{2itb} we infer that every $h\in\Delta(A)$ is continuous. The
\emph{spectrum} $\sigma(a)$ of $a\in A$ denotes the set of all
$\lambda\in\setC$ such that $a-\lambda u$ is \emph{not} invertible. Then
\itemref{2itc} relates the spectrum $\sigma(a)$ to the space $\Delta(A)$ of
all nontrivial multiplicative linear functionals $h\colon A\to\setC$ by
\begin{equation*}
\sigma(a)=\{h(a):h\in\Delta(A)\}.
\end{equation*}
In particular, $a\in A$ is invertible if and only if $0\notin\sigma(a)$.
This suggests to identify the invertible elements of $A$ by determining all
nontrivial multiplicative linear functionals $h\in\Delta(A)$.

The following lemma taken from Edwards \cite[11.4.5]{Ed1957} (see also Rudin
\cite[Exercise 11.5]{Ru1974}) generalizes Theorem \ref{T2}\,\ref{2itc}) by
replacing the inversion map with a holomorphic function.

\begin{clemma}
Let $a\in A$, and let $f$ be a holomorphic function defined on a region
$G\subseteq\setC$ such that $\sigma(a)\subseteq G$. Then there exists an
element $b\in A$ such that $h(b)=f(h(a))$ for every $h\in\Delta(A)$.
\end{clemma}

%%%%%%%%%%%%%%%%%%%%%%%%%%%%%%%%%%%%%%%%%%%%%%%%%%%%%%%%%%%%%%%%%%%%%%%%%%%%%

\section{Weighted semigroup algebras and their spectra} \label{algebras}

%%%%%%%%%%%%%%%%%%%%%%%%%%%%%%%%%%%%%%%%%%%%%%%%%%%%%%%%%%%%%%%%%%%%%%%%%%%%%

\noindent
Let $(\Lambda,\cdot)$ be an arbitrary semigroup with unit element~$e$.
For the moment, multiplicative notation is used, as~$\Lambda$ need not be
commutative. We call a function $w\colon\Lambda\to (0,\infty)$ a
\emph{weight function} if $w(e)=1$ and~$w$ is submultiplicative, i.e.
$w(\lambda\lambda')\leq w(\lambda)w(\lambda')$ for all
$\lambda,\lambda'\in\Lambda$. Then a Banach algebra $\scra_w(\Lambda)$ can be
defined, which is a weighted analogue of the Banach algebra
$\ell_1(\Lambda)$ introduced by Hewitt and Zuckerman~\cite{HZ1956}
(see \cite{DD2009}; cf.~\cite[p.\,70]{Ne2000} for the case of semigroups with
involution). In particular, the construction applies to additive
subsemigroups of $[0,\infty)$ (as discussed in the introduction), or additive
subsemigroups of $[0,\infty)^r$ (as needed in connection with
multidimensional Dirichlet series).

Recall that the sum $\sum_{i\in I}a_i$ of a family $(a_i)_{i\in I}$ of
numbers $a_i\in [0,\infty)\cup\{\infty\}$ is defined as the supremum of the
sums $\sum_{i\in F}a_i$ over finite index sets $F\subseteq I$. A family
$(a_i)_{i\in I}$ of complex numbers is called \emph{absolutely summable} if
$\sum_{i\in I}|a_i|<\infty$. Then $a_i\neq 0$ for only countably many
$i\in I$, and the net of finite partial sums $\sum_{i\in F}a_i$ converges.
Its limit is denoted by $\sum_{i\in I}a_i$
(cf.~\cite[Chapter V, \S3]{Di1969}, \cite[Chapter III, Exercise 23]{SW1999}).
Now the weighted semigroup algebras are obtained as follows
(\cite{DD2009}; cf.~\cite{Ne2000}):

\begin{proposition} \label{P1}
Let $(\Lambda,\cdot)$ be a semigroup with a unit element and~$w$ be a weight
function on~$\Lambda$. Let $\scra_w(\Lambda)$ be the set of all families
$a=(a(\lambda))_{\lambda\in \Lambda}$ of complex numbers such that
$\|a\|_w:=\sum_{\lambda\in \Lambda}|a(\lambda)|w(\lambda)<\infty$. If
$a,b\in\scra_w(\Lambda)$, then the following holds:
\begin{enumeratea}
\item \label{ita}
For each $\lambda\in\Lambda$ the numbers $a(\lambda')b(\lambda'')$ for
$(\lambda',\lambda'')\in \Lambda\times \Lambda$ with
$\lambda'\lambda''=\lambda$ form an absolutely summable family. Thus
$c(\lambda):=\sum_{\lambda'\lambda''=\lambda}a(\lambda')b(\lambda'')$ is
defined.
\item \label{itb}
The family $a*b:=c$ is in $\scra_w(\Lambda)$.
\end{enumeratea}
The multiplication $*$ makes $(\scra_w(\Lambda),\|\DOT\|_w)$ a unital
Banach algebra.
\end{proposition}

For the remainder of this section we return to the additive notation and let
$(\Lambda,+)$ be a commutative semigroup with neutral element~$0$. By a
\emph{character} of~$\Lambda$ we mean a homomorphism $\psi$ of
$(\Lambda,+)$ to $(\setC,\CDOT)$ such that $\psi(0)=1$. If $w$ is a weight
function on $\Lambda$, we let $\wh{\Lambda}_w$ be the set of all characters
$\psi$ of $\Lambda$ which are \emph{$w$-bounded} in the sense that
\begin{equation*}
|\psi(\lambda)|\leq w(\lambda)\quad\text{for all $\lambda\in\Lambda$}
\end{equation*}
(cf.~\cite{BC1984} and \cite{Ne2000} for the case of semigroups with
involution).

Complex homomorphisms of $\scra_w=\scra_w(\Lambda)$ and $w$-bounded
characters of $\Lambda$ are closely related. To see this, we use again
Kronecker's $\delta$ to define an element
$\delta_\lambda\colon\Lambda\to\setC$ with
$\mu\mapsto\delta_{\lambda,\mu}$ in $\scra_w$, such that
$\|\delta_\lambda\|_w=w(\lambda)$. If $h\in\Delta(\scra_w)$, then
\begin{equation*}
\psi_h\colon (\Lambda,+)\to (\setC,\cdot),\quad
\psi_h(\lambda):=h(\delta_\lambda)
\end{equation*}
is a homomorphism, for $\delta_0=\ve$ and $\delta_{\lambda+\lambda'}=
\delta_\lambda*\delta_{\lambda'}$. Since
$|\psi_h(\lambda)|=|h(\delta_\lambda)|\leq\|\delta_\lambda\|_w=w(\lambda)$,
the character $\psi_h$ is $w$-bounded. Conversely, let
$\psi\in\wh{\Lambda}_w$. Then the family
$(a(\lambda)\psi(\lambda))_{\lambda\in\Lambda}$ is absolutely summable for
each $a\in \scra_w$, as
\begin{equation} \label{esta}
 \sum_{\lambda\in\Lambda}\,|a(\lambda)\,\psi(\lambda)| \leq
 \sum_{\lambda\in\Lambda}\,|a(\lambda)|\,w(\lambda)=\|a\|_w<\infty.
\end{equation}
We can therefore define a function $h_\psi\colon\scra_w\to\setC$ via
\begin{equation} \label{eq2}
h_\psi(a) := \sum_{\lambda\in\Lambda}\,a(\lambda)\,\psi(\lambda)
\quad\text{for all }a\in\scra_w.
\end{equation}
Then $h_\psi$ is linear and of operator norm $\leq 1$ by \eqref{esta},
hence continuous. In fact, $h_\psi\in\Delta(\scra_w)$.
To see this, it remains to show that $h_\psi(a*b)=h_\psi(a)h_\psi(b)$ for all
$a,b\in\scra_w$. It suffices to assume that $a=\delta_\lambda$ and
$b=\delta_{\lambda'}$ with $\lambda,\lambda'\in \Lambda$ (as such elements
span a dense vector subspace of $\scra_w$). But then
$h_\psi(a*b)=\psi(\lambda\lambda')=\psi(\lambda)\psi(\lambda')
=h_\psi(a)h_\psi(b)$ indeed. We readily deduce:

\begin{proposition}\label{spectrum}
Let $(\Lambda,+)$ be a commutative semigroup with neutral element~$0$,
and $w$ be a weight function on~$\Lambda$. Let $\scra_w=\scra_w(\Lambda)$.
Then the map
\begin{equation}
\Delta(\scra_w)\to\wh{\Lambda}_w,\quad h\mapsto \psi_h
\end{equation}
is a bijection, with inverse $\psi\mapsto h_\psi$.
\end{proposition}

We say that a weight function $w$ on a commutative semigroup $(\Lambda,+)$
is \emph{admissible}, if it satisfies the conditions a) and b) described in
the introduction. We write $\scrw(\Lambda)$ (or simply $\scrw$) for the set
of all admissible weight functions on $\Lambda$.

\begin{lemma}\label{admissible}
Let $w$ be an admissible weight function on a commutative semigroup
$(\Lambda,+)$ with neutral element~$0$. Let $\psi$ be a character of
$\Lambda$. Then $\psi$ is $w$-bounded if and only if $\psi$ is a bounded
function.
\end{lemma}

\begin{proof}
If $\psi$ is $w$-bounded, then $|\psi(\lambda)|=\sqrt[k]{|\psi(kn)|}\leq
\sqrt[k]{w(k\lambda)}\to 1$ as $k\to\infty$, and thus $|\psi(\lambda)|\leq 1$
for each $\lambda\in\Lambda$. Conversely, assume that $\psi$ is bounded, say
$|\psi(\lambda)|\leq C$ with $C>0$. Then
$|\psi(\lambda)|=\sqrt[k]{|\psi(k\lambda)|}\leq \sqrt[k]{C}$
for each $k\in\setN$ and thus $|\psi(\lambda)|\leq 1\leq w(\lambda)$, using
that $\sqrt[k]{C}\to 1$.
\end{proof}

For future use, we write $\setU:=\{z\in\setC:|z|<1\}$. If $\psi$ is a
bounded character, then $\psi$ only takes values in the closed unit disk
$\ol{\setU}$.

%%%%%%%%%%%%%%%%%%%%%%%%%%%%%%%%%%%%%%%%%%%%%%%%%%%%%%%%%%%%%%%%%%%%%%%%%%%%%

\section{Two proofs of Theorem \ref{T1}} \label{S4}

%%%%%%%%%%%%%%%%%%%%%%%%%%%%%%%%%%%%%%%%%%%%%%%%%%%%%%%%%%%%%%%%%%%%%%%%%%%%%

\noindent
According to Theorem \ref{T2} and the definition of the spectrum it suffices
to show that $0\notin\sigma(a)$ for $a\in\scra_w$ or, equivalently,
$h(a)\neq 0$ for all $h\in\Delta(\scra_w)$. To enable this, it is useful
to have a description of functionals $h\in\Delta(\scra_w)$. Combining
Proposition~\ref{spectrum} and Lemma~\ref{admissible} we get the following
lemma (to be found in Lucht and Reifenrath \cite{LR1996} for discrete
$\Lambda\subseteq [0,\infty)$):

\begin{rlemma}
Let $(\Lambda,+)$ be a commutative semigroup with identity~$0$, and $w$ be
an admissible weight function on~$\Lambda$. Then to every
$h\in\Delta(\scra_w)$ there corresponds a bounded character $\psi$ of
$\Lambda$ such that
\begin{equation} \label{eq2'}
h(a) = \sum_{\lambda\in\Lambda}\,a(\lambda)\,\psi(\lambda)
       \quad\text{for all }a\in\scra_w.
\end{equation}
Conversely, every bounded character $\psi\colon\Lambda\to\setC$ determines a
unique $h\in\Delta(\scra_w)$ with \eqref{eq2'}.
\end{rlemma}

First we derive Theorem \ref{T1} from the special case $w=1$, due to Edwards
\cite{Ed1957}.

\begin{proof}[First proof of Theorem \ref{T1}]
Since $\scra_w$ is a Banach subalgebra of $\scra_1$ and the characterization
\eqref{eq2'} is independent of $w$, the Representation Lemma shows that
\begin{equation*}
 \scra^*_w = \bigg\{a\in\scra_w:
 \sum_{\lambda\in\Lambda}\,a(\lambda)\,\psi(\lambda)\neq 0\,
 \text{~for all characters~} \psi \text{~of~} \Lambda\big\}.
\end{equation*}
Thus $\scra^*_w=\scra^*_1\cap\scra_w$. By Edwards' theorem \cite{Ed1957},
\begin{equation*}
 \scra^*_1 =
 \big\{a\in\scra_1: 0\notin\ol{\wt{a}(\setH)}\big\},
\end{equation*}
and the fact that $\wt{a}$ does not depend on $w$, the assertion follows.
\end{proof}

A different option for proving Theorem \ref{T1} without recourse to the
Theorem of Edwards is based on a topological linkage of the image set
$\wt{a}(\setH)$ and the spectrum $\sigma(a)$ of functions $a\in\scra_w$,
namely an approximation of the functions $h\in\Delta(\scra_w)$ by the
functions $h_s\in\Delta(\scra_w)$ associated with the specific characters
$\lambda\mapsto\psi_s(\lambda)=e^{-\lambda s}$ for $s\in\ol{\setH}$.

\begin{dlemma} \label{LII.3.2}
Let $\Lambda\subseteq[0,\infty)$ be an additive semigroup with $0\in\Lambda$,
and let $w\in\scrw$. Then for any $a\in\scra_w$ the set $\wt{a}(\setH)$ is
dense in $\sigma(a)$.
\end{dlemma}

The Density Lemma yields the announced alternative proof of Theorem \ref{T1}.

\begin{proof}[Second proof of Theorem \ref{T1}]
According to Theorem \ref{T2}\,\ref{2itc}) the invertibility of $a$ in
$\scra_w$ follows from $\sigma(a)\subseteq\ol{\wt{a}(\setH)}$ and
$0\notin\ol{\wt{a}(\setH)}$.
\end{proof}

It remains to verify the Density Lemma.

%%%%%%%%%%%%%%%%%%%%%%%%%%%%%%%%%%%%%%%%%%%%%%%%%%%%%%%%%%%%%%%%%%%%%%%%%%%%%

\section{Proof of the Density Lemma, part 1} \label{S5}

%%%%%%%%%%%%%%%%%%%%%%%%%%%%%%%%%%%%%%%%%%%%%%%%%%%%%%%%%%%%%%%%%%%%%%%%%%%%%

\noindent
First we establish a special case of the Density Lemma, assuming, in
addition, that the semigroup $\Lambda$ is \emph{free}, i.e. $\Lambda$ has a
free generating set~$\Beta$.

Let $a\in\scra_w$, $\vt>0$, and $h\in{\Delta}(\scra_w)$. Denote by~$\psi$ the
bounded character of~$\Lambda$ associated with~$h$, as in the Representation
Lemma. Then $\psi(\Lambda)\subseteq\ol{\setU}$. In order to show that there
exists an $s\in\setH$ satisfying
\begin{equation*}
|h_s(a)-h(a)|<3\vt,
\end{equation*}
it suffices to verify that, for any finite subset $\Gamma\subseteq \Lambda$,
the estimate
\begin{equation} \label{eq5}
\bigg|\sum_{\lambda\in\Gamma}\,a(\lambda)\,e^{-\lambda s} -
      \sum_{\lambda\in\Gamma}\,a(\lambda)\,\psi(\lambda)\bigg|<\vt
\end{equation}
holds with suitably chosen $s\in\setH$. In fact, there exists a finite subset
$\Gamma\subseteq\Lambda$ such that
$\sum_{\lambda\in\Lambda\setminus\Gamma}w(\lambda)|a(\lambda)|<\vt$
so that
\begin{equation*}
 \bigg|\sum_{\lambda\in\Lambda\setminus\Gamma} a(\lambda)\,e^{-\lambda s}
 -\sum_{\lambda\in\Lambda\setminus\Gamma} a(\lambda)\,\psi(\lambda)\bigg|
 \leq 2\,\sum_{\lambda\in \Lambda\setminus \Gamma}\,|a(\lambda)|\leq 2\vt
\end{equation*}
for each $s\in\setH$, from which the assertion follows.

Let $\frak{b}=(\beta_1,\ldots,\beta_k)$ consist of generators
$0<\beta_1,\ldots,\beta_k\in\Beta$ such that every $\lambda\in\Gamma$
can be expressed in the form
\begin{equation}\label{inspan}
\lambda=\sum_\kappa\,\nu_\kappa\beta_\kappa
\end{equation}
with $\nu_\kappa\in\setN_0$ for $1\leq\kappa\leq k$. Then
\begin{equation*}
P(\frak{z})=\sum_{\lambda\in\Gamma}\,a(\lambda)\,e^{-\lambda s}
\end{equation*}
can be regarded as a polynomial of $\frak{z}=(z_1,\ldots,z_k)\in\ol{\setU}^k$
with variables $z_\kappa=e^{-\beta_\kappa s}$ for $1\leq\kappa\leq k$. Now
\eqref{inspan} leads to
\begin{equation*}
\psi(\lambda)=\prod_{1\leq\kappa\leq k}\,\psi(\beta_\kappa)^{\nu_\kappa}
\end{equation*}
and
\begin{equation*}
\sum_{\lambda\in\Gamma}\,a(\lambda)\,\psi(\lambda) =
P\big(\psi(\beta_1),\ldots,\psi(\beta_k)\big).
\end{equation*}
Since $|\psi(\beta_\kappa)|\leq 1$ for all $\kappa$, the following
generalized version of a lemma of Spilker and Schwarz
\cite[Hilfssatz 5.1]{SS1979} (see also Hewitt and Williamson
\cite[Lemma 2]{HW1957}) yields the existence of $s\in\setH$ with \eqref{eq5},
which establishes the Density Lemma as well as Theorem \ref{T1} in the
special case of free semigroups $\Lambda$.

\begin{lemma} \label{L1}
For $k\in\setN$ let $\frak{v}\colon\ol{\setH}\to\ol{\setU}^k$ be defined by
\begin{equation*}
\frak{v}(s)=e^{-\frak{b} s}:=(e^{-\beta_1 s},\ldots,e^{-\beta_k s})
\end{equation*}
with $\setQ$-linearly independent numbers $\beta_\kappa>0$ for
$1\leq\kappa\leq k$. If $f\colon\ol{\setU}^k\to\setC$ is continuous and
holomorphic on $\setU^k$, then $f(\frak{v}(\ol{\setH}))$ is a dense
subset of $f(\ol{\setU}^k)$.
\end{lemma}

\begin{proof}
The assertion is trivial for constant functions
$f\colon\ol{\setU}^k\to\setC$.

Let $f$ be nonconstant. It suffices to show that for any $\vt>0$ there is
some $s\in\ol{\setH}$ such that $|g(\frak{v}(s))|<\vt$, where
$g(\frak{z})=f(\frak{z})-c$ for $\frak{z}=(z_1,\ldots,z_k)\in\ol{\setU}^k$
with an arbitrary $c\in f(\ol{\setU}^k)$. Suppose, to the contrary, that
there is some $\vt>0$ such that
\begin{equation} \label{eq6}
|g(\frak{z})| \geq \vt \quad\text{for all }\frak{z}\in\frak{v}(\ol{\setH}).
\end{equation}

First we show that for any $\vt>0$ there is some $t\in\setR$ such that
\begin{equation} \label{eq7}
\big|z_\kappa-e^{-\beta_\kappa(\sigma+it)}\big|<\vt
 \qquad (\kappa=1,\ldots,k),
\end{equation}
if $|z_\kappa|=e^{-\beta_\kappa \sigma}$ for $\kappa=1,\ldots,k$. It suffices
to verify \eqref{eq7} for $\sigma=0$ (i.e., $|z_k|=1$)\,: The Kronecker
approximation theorem (cf. Hardy and Wright \cite[Theorem 444]{HW2008})
applied to the $\setQ$-linearly independent set $\{\beta_1,\ldots,\beta_k\}$
entails that for any $\vt>0$ there exist numbers $t\in\setR$ and
$m_1,\ldots,m_k\in\setZ$ satisfying
\begin{equation*}
\Big|\frac{t}{2\pi}\,\beta_\kappa-m_\kappa-\frac{\arg{z_\kappa}}{2\pi}\Big|
 < \frac{\vt}{2} \qquad(\kappa=1,\ldots,k).
\end{equation*}
For $\vt>0$ sufficiently small, it follows that
\begin{equation*}
 \big|e^{i\arg{z_\kappa}}-e^{it\beta_\kappa}\big| < \vt,
 \qquad(\kappa=1,\ldots,k),
\end{equation*}
as asserted. Hence, for any $\sigma\geq 0$, $\frak{v}(\sigma+i\setR)$ is a
dense subset of the poly-circle
\begin{equation*}
e^{-\frak{b}\sigma}\pa{\setU} :=
 \{\frak{z}\in\setC^k:|z_\kappa|=e^{-\beta_\kappa\sigma}~\text{ for }~
 1\leq\kappa\leq k\},
\end{equation*}
and
\begin{equation} \label{eq8}
 |g(\frak{z})|\geq\vt \quad \text{for all }\,
 \frak{z}\in e^{-\frak{b}\sigma}\pa{\setU}.
\end{equation}

As $\frak{o}=(0,\ldots,0)\in\ol{\setU}^k$ is limit point of
$\frak{v}(\ol{\setH})$, we have $|g(\frak{o})|\geq\vt$. Therefore we may
define
\begin{equation*}
\sigma_0
 = \inf{\big\{\sigma\geq 0: |g(\frak{z})|\geq\ts\frac{1}{2}\,\vt
   \text{ for all }\frak{z}\in e^{-\frak{b}\sigma}\ol{\setU}\big\}},
\end{equation*}
where $e^{-\frak{b}\sigma}\ol{\setU}:=e^{-\beta_1 \sigma}\ol{\setU}\times
\cdots\times e^{-\beta_k \sigma}\ol{\setU}$.
Then $g$ has no zero in the compact poly-disc
$e^{-\frak{b}\sigma_0}\ol{\setU}$, and $1/g$ represents a continuous
function on $e^{-\frak{b}\sigma_0}\ol{\setU}$ that is holomorphic on
$e^{-\frak{b}\sigma_0}\setU$. By applying the maximum principle and the
Weierstraß convergence theorem for holomorphic functions to each component
of the poly-disc $e^{-\frak{b}\sigma_0}\ol{\setU}$, we obtain from
\eqref{eq8} that
\begin{equation*}
 \max{\bigg\{\Big|\frac{1}{g(\frak{z})}\Big| :
 \frak{z}\in e^{-\frak{b}\sigma_0}\ol{\setU}\bigg\}} =
 \max{\bigg\{\Big|\frac{1}{g(\frak{z})}\Big| :
 \frak{z}\in e^{-\frak{b}\sigma_0}\pa{\setU}\bigg\}} \leq \frac{1}{\vt}\,.
\end{equation*}
Hence
$|g(\frak{z})|\geq\vt$ for all $\frak{z}\in e^{-\frak{b}\sigma_0}\ol{\setU}$.

This gives the desired contradiction, as $e^{-\frak{b}\sigma_0}\ol{\setU}$
contains some point $\frak{z}$ with $|g(\frak{z})|<\vt$, in case of
$\sigma_0>0$ by the definition of $\sigma_0$, and in case of $\sigma_0=0$ by
$c\in f(\ol{\setU}^k)$.
\end{proof}

%%%%%%%%%%%%%%%%%%%%%%%%%%%%%%%%%%%%%%%%%%%%%%%%%%%%%%%%%%%%%%%%%%%%%%%%%%%%%

\section{Proof of the Density Lemma, part 2} \label{S6}

%%%%%%%%%%%%%%%%%%%%%%%%%%%%%%%%%%%%%%%%%%%%%%%%%%%%%%%%%%%%%%%%%%%%%%%%%%%%%

\noindent
To complete the proof of the Density Lemma we have to remove the assumption
on $\Lambda$ to be free. As in the preceding section, for $\vt>0$ and a
finite subset $\Gamma\subseteq\Lambda$, we merely need to find $s\in\setH$
with \eqref{eq5}. Since only the values of $\psi$ on~$\Gamma$ enter
\eqref{eq5}, the next lemma allows $\Lambda$ to be replaced with a free
semigroup $[\Beta]$, to which the special case from Section~\ref{S5} applies.

We shall use the following terminology and facts concerning convex cones:
A subset~$C$ of a finite-dimensional real vector space~$W$ is called a
\emph{convex cone} if $C$ is convex and $[0,\infty)\cdot C\subseteq C$.
Then~$C$ is a semigroup under addition. If $C\not=\emptyset$, then $C-C$
is a vector subspace of~$W$ and $C$ has non-empty interior in~$C-C$
(cf. \cite[Proposition~V.1.4\,(ii)]{Ne2000}).
The \emph{dimension of~$C$} is defined as the dimension of $C-C$.
A convex cone $F\subseteq C$ is called a \emph{face of $C$} if $x+y\in F$
for elements $x,y\in C$ implies $x,y\in F$. A face of the form $[0,\infty)x$
with $x\not=0$ is called an \emph{extreme ray} of~$C$.
A convex cone $C\subseteq W$ is called \emph{polyhedral} if it is generated
by a finite set $E\subseteq W$ (i.e., $C$ is of the form~(\ref{eqcone}) 
below).%
\footnote{See \cite[Chapter V.1]{Ne2000} for basic facts on polyhedral cones,
as first spelled out in \cite{We1936}.}
Then every face $F\subseteq C$ is generated by $F\cap E$, as is well-known
(cf. \cite[Theorems~7.2 and 7.3]{Br1983}).%
\footnote{Omitting only a trivial case, consider a non-zero element $x\in F$.
Then $x=r_1x_1+\cdots+r_kx_k$ with elements $x_1,\ldots, x_k\in E$ and
$r_1,\ldots, r_k>0$. Since $F$ is a face, it follows that
$r_1x_1,\ldots, r_kx_k\in F$ and hence $x_1,\ldots, x_k\in F$.}
In particular, every extreme ray of~$C$ is of the form $[0,\infty)x$ with
some $x\in E$.

\begin{elemma}
Let $\psi\colon\Lambda\to\setC$ be a bounded character of an additive
semigroup $\Lambda\subseteq[0,\infty)$ with $0\in\Lambda$, and let
$\Gamma\subseteq\Lambda$ be a finite subset. Then there exists a
$\setQ$-linearly independent set
$\Beta=\{\beta_1,\ldots,\beta_k\}\subseteq(0,\infty)$ such that
\begin{equation*}
\Gamma\subseteq [\Beta]:=\setN_0\beta_1+\cdots+\setN_0\beta_k,
\end{equation*}
and a bounded character $\phi\colon[\Beta]\to\setC$ which coincides with $\psi$
on~$\Gamma$.
\end{elemma}

The proof requires some notional arrangements.

Let $\setQ^+=\{q\in\setQ:q>0\}$ and $\setQ^+_0=\setQ^+\cup\{0\}$. Given a
$\setQ$-vector space $V$,
we let $V_\setR=\setR\otimes_\setQ V$ be the $\setR$-vector space%
\footnote{with the $\setR$-basis $\{1\}\otimes B$ where $B$ is a
$\setQ$-basis of $V$.}
associated with the $\setQ$-vector space $V$
(cf. \cite[Chapt.\,XVI,\,\S\,4]{La2002}). If $E=\{x_1,\ldots,x_k\}$ is a
finite subset of $V$, we write $[E]:=\setN_0 x_1+\cdots+\setN_0 x_k$ for
the submonoid of $(V,+)$ generated by $E$,
\begin{equation*}
 \conv_\setQ(E):= \big\{q_1 x_1+\cdots+q_k x_k:
 q_1,\ldots,q_k\in\setQ_0^+,~q_1+\cdots+q_k=1\big\}
\end{equation*}
for its rational convex hull, $\conv_\setR(E)$ for its usual real convex
hull, and
\begin{equation}\label{eqcone}
 \cone_\setR(E):= [0,\infty)\cdot\conv_\setR(E)
\end{equation}
for the convex cone generated by $E$.

The dual space $V^*$ of a $\setQ$-vector space $V$ consists of all
$\setQ$-linear functionals from $V$ to $\setQ$. Given a subset
$T\subseteq V$, we define its \emph{rational dual cone} by
\begin{equation*}
T^\star:=\{\rho\in V^*:\rho(T)\subseteq[0,\infty)\}.
\end{equation*}
Then, by identifying $V$ with $(V^*)^*$,
$T\subseteq (T^\star)^\star$. If $T\subseteq U$, then
$T^\star\supseteq U^\star$. If $v_1,\ldots,v_k\in V$ is a basis and
$v_1^*,\ldots,v_k^*\in V^*$ its dual basis, $v_i^*(v_j)=\delta_{ij}$, then
\begin{equation} \label{eq9}
\{v_1,\ldots,v_k\}^\star={\setQ_0^+}v_1^*+\cdots+{\setQ_0^+}v_k^*.
\end{equation}
We let $V^*_\setR$ be the real dual space of $V_\setR$ and define the
\emph{real dual cone}
\begin{equation*}
T_\setR^\star:=\{\rho\in V^*_\setR:\rho(T)\subseteq[0,\infty)\}
\end{equation*}
of a subset $T\subseteq V_\setR$. As usual, we identify $V^*_\setR$ with
$(V^*)_\setR$. For $T$ finite and $\rho\in V^*_\setR$, $\rho$ is in the
interior of $T_\setR^\star$ if and only if $\rho(t)>0$ for all $t\in T$,
$t\neq 0$.

\begin{proof}[Proof of the Extension Lemma]
We may assume that $0\notin\Gamma\neq\emptyset$ and, after shrinking
$\Lambda$ to $[\Gamma]$, that $\Lambda=[\Gamma]$. Let $V=\Span_\setQ(\Gamma)$
denote the $\setQ$-linear space spanned by $\Gamma$, and define
\begin{equation*}
 \Gamma' :=\{\gamma\in\Gamma:\psi(\gamma)\neq 0\},\quad
 \Gamma_0:=\{\gamma\in\Gamma:\psi(\gamma)=0\},
\end{equation*}
and $V':=\Span_\setQ(\Gamma')$.

\begin{step} \label{step1}
Write $\psi_1:=|\psi|$ and $\psi_2(\xi):=\psi(\xi)/\psi_1(\xi)$ for
$\xi\in[\Gamma']$. By the theorem in Ross \cite{Ro1961}, $\psi_2$ extends to
a character $\vp_2\colon (V,+)\to(\setC,\,\cdot\,)$ with values in the circle
group $\pa{\setU}$. If we can extend $\psi_1$ to a bounded character $\vp_1$
on $[B]$ for suitable $B$, then $\vp_1\vp_2$ extends $\psi$. Hence
$\psi(\Gamma)\subseteq[0,\infty)$ without loss of generality.
\end{step}

\begin{step} \label{step2}
Since $(0,\infty)$ is a divisible, torsion-free abelian group under
multiplication, the homomorphism
$\psi|_{[\Gamma']}\colon [\Gamma']\to (0,\infty)$ extends uniquely to a
homomorphism of groups $\vt\colon V'\to (0,\infty)$ (cf. \cite[A7]{HR1979}).%
\footnote{In a first step, extend $\psi|_{[\Beta']}$ to a group homomorphism
$[\Beta']-[\Beta']\to (0,\infty)$ via
$\xi_1-\xi_2\mapsto \psi(\xi_1)/\psi(\xi_2)$.}
Then $-\ln\circ\,\vt\colon V'\to\setR$ is a $\setQ$-linear map and thus
extends uniquely to an $\setR$-linear functional
$\rho\colon V'_\setR\to\setR$. Note that
\begin{equation*}
\vt(\xi)=e^{-\rho(\xi)} \quad\text{for all }\,\xi\in V'
\end{equation*}
by construction of $\rho$. Since $|\vt(\xi)|=|\psi(\xi)|\leq 1$, we have
$\rho(\xi)\geq 0$ for all $\xi\in\Gamma'$ and thus
$\rho\in (\Gamma')_\setR^\star$.
\end{step}

\begin{step} \label{step3}
We claim that
\begin{equation} \label{eq10}
\conv_\setQ(\Gamma_0)\cap V'=\emptyset.
\end{equation}
Suppose, to the contrary, that
$\eta\in\conv_\setQ(\Gamma_0)\cap V'\neq\emptyset$. Then there exist numbers
\mbox{$k,\ell,t\in\setN$} with $k\leq\ell$, elements
$\eta_1,\ldots,\eta_t\in\Gamma_0$, $\xi_1,\ldots,\xi_\ell\in\Gamma'$ and
coefficients $q_1,\ldots,q_t,r_1,\ldots,r_\ell\in\setQ^+$ such that
\begin{equation*}
 0\neq\eta=\sum_{1\leq\tau\leq t}\,q_\tau\,\eta_\tau=
 \sum_{1\leq\lambda\leq k}\,r_\lambda\,\xi_\lambda -
 \sum_{k<\lambda\leq \ell}\,r_\lambda\,\xi_\lambda.
\end{equation*}
By multiplying the equation with the common denominator $a\in\setN$ of the
rational coefficients $q_\tau$, $r_\lambda$ for $\tau\leq t$,
$\lambda\leq\ell$, we obtain
\begin{equation*}
 a\,\eta+\sum_{k<\lambda\leq \ell}\,b_\lambda\,\xi_\lambda=
 \sum_{1\leq\lambda\leq k}\,b_\lambda\,\xi_\lambda
\end{equation*}
with certain coefficients $b_\lambda\in\setN$ and $0\not=a\eta\in[\Gamma_0]$.
Then
\begin{equation*}
 \psi(a\eta)\cdot \psi(\xi_{k+1})^{b_{k+1}}\cdots\psi(\xi_\ell)^{b_\ell}=
 \psi(\xi_1)^{b_1}\cdots\psi(\xi_k)^{b_k},
\end{equation*}
which is a contradiction since $\psi(a\eta)=0$ but
$\psi(\xi_\lambda)\neq 0$ for $1\leq\lambda\leq k$.
\end{step}

\begin{step} \label{step4}
Let $U:=V/V'$ and $\pi\colon V\to U$ be the quotient map. From Step~3 we
know that
$0\notin\pi{(\conv_\setQ{(\Gamma_0)})} = \conv_\setQ{(\pi{(\Gamma_0)})}$.
Hence the Separation Lemma (see Section \ref{S7}) provides a
$\setQ$-linear functional $\chi\in U^*$ such that $\chi(\alpha)>0$ for
each $\alpha\in\pi(\Gamma_0)$. Then, with $\theta:=\chi\circ\pi\in V^*$,
we obtain
\begin{equation} \label{eq11}
 \theta(\alpha)>0\,\text{ for each }\,\alpha\in\Gamma_0\,\text{ and }\,
 \theta|_{\Gamma'}=0,
\end{equation}
whence $\theta\in\Gamma^\star$ in particular.
\end{step}

\begin{step} \label{step5}
Let $\rho'\in V^*_\setR$ be any real functional such that
$\rho'|_{V'_\setR}=\rho$. We choose $c>0$ so large that
$\zeta:=\rho'+c\theta\in\Gamma^\star_\setR$. This is possible since
\begin{equation} \label{eq12}
\zeta|_{\Gamma'}=\rho'|_{\Gamma'}\geq 0
\end{equation}
(as $\theta$ vanishes on $\Gamma'$), and furthermore
\begin{equation*}
\zeta(\alpha)=\rho'(\alpha)+c\theta(\alpha)
\end{equation*}
for $\alpha\in\Gamma_0$, which can be made arbitrarily large since
$\theta(\alpha)>0$.
\end{step}

\begin{step} \label{step6}
We have $0\notin\conv_\setR(\Gamma)$ in $V_\setR$, because $\Gamma$ is a
subset of the convex set $(0,\infty)$. Hence $P:=\cone_\setR(\Gamma)$ is a
pointed%
\footnote{That means, $P$ does not contain lines.}
polyhedral cone in $V_\setR$ whose extreme rays are of the form
$[0,\infty)\alpha$ for certain $\alpha\in\Gamma$ (as recalled above).
Since $\Span_\setR(\Gamma)=V_\setR$, the cone $P$ has non-empty interior.
Hence also $P^\star_\setR=\Gamma^\star_\setR\subseteq V^*_\setR$ is a pointed
polyhedral cone with non-empty interior (see Neeb
\cite[Propositions V.1.5 (ii) and V.1.21]{Ne2000}). It is known from the
theory of polyhedral cones in real vector spaces that the extreme rays of
$P^\star_\setR$ are of the form $[0,\infty)\alpha$ for a functional
$\alpha\in V^*_\setR$ such that $F:=\ker{\alpha}\cap P$ is a codimension~$1$
face of~$P$ (see, e.g. \cite[Theorem~3]{Ta1976}). We claim that
\emph{$\alpha$ can be chosen in $V^*$.}

To see this, recall first that $F=\cone_\setR(\Gamma'')$ for some subset
$\Gamma''\subseteq\Gamma$, and $F$ has non-empty interior in $\ker{\alpha}$.
Hence $\ker{\alpha}=\Span_\setR(F)=(\Span_\setQ(\Gamma''))_\setR$ is defined
over $\setQ$. After replacing $\alpha$ with a positive real multiple to
ensure that $\alpha(V)\subseteq\setQ$, we have $\alpha\in V^*$, as
desired.

Consequently, $F\cap V^*$ is dense in $F$ for each face $F$ of
$P^\star_\setR$. Furthermore, $\Span_\setQ(\Gamma^\star)=V^*$.
\end{step}

\begin{step}\label{step7}
Let $F\subseteq \Gamma^\star_\setR$ be a face of dimension~$\ell\geq 1$,
and $\algint(F)$ be its interior relative $\aff_\setR(F)=\Span_\setR(F)$. We
show: \emph{For every $\eta\in\algint(F)$ there exists a $\setQ$-basis
$b_1,\ldots,b_\ell\in F\cap V^*$ of $\Span_\setQ(F\cap V^*)$ with
$\eta\in\cone_\setR(b_1,\ldots,b_\ell)$. Moreover, $b_1$ can be chosen as an
arbitrary non-zero vector in $F\cap V^*$.}

In fact, if $\eta=0$, we can simply select a $\setQ$-basis from generators
$\alpha\in V^*$ for extreme rays of~$F$, which exist by Step~\ref{step6}
(or extend a given vector~$b_1$ by such vectors to a basis). If $\eta\not=0$
(which we assume now), we proceed by induction on~$\ell$:

If $\ell=1$, then $F=[0,\infty) \alpha$ for some
$\alpha\in\Gamma^\star\subseteq V^*$ (see Step~\ref{step6}). We can now take
$b_1:=\alpha$ (or any prescribed non-zero vector in $F\cap V^*$).

Induction step.
There exists an~$x$ in the interior $P^0$, such that $\eta(x)>0$.
Since~$V$ is dense in~$V_\setR$, we may assume that $x\in V$. After passage
to a positive multiple of~$\eta$, we may also assume that $\eta(x)\in\setQ$.
Then $\gamma(x)>0$ for all $\gamma\in\Gamma^\star_\setR\setminus\{0\}$. Hence
$K:=\{\gamma\in\Gamma^\star_\setR:\gamma(x)=\eta(x)\}$ is a closed convex set
such that $\Gamma^\star_\setR=[0,\infty)K$ and $\eta\in K$. Choose
$\alpha_1,\ldots,\alpha_n\in\Gamma^\star_\setR$ such that
$[0,\infty)\alpha_j$, $j=1,\ldots,n$, are the extreme rays of
$\Gamma^\star_\setR$; after passage to positive multiples, we may assume that
$\alpha_j(x)=\eta(x)$ for all~$j$. If $r_1,\ldots,r_n\geq 0$ and
$\gamma:=r_1\alpha_1+\cdots+ r_n\alpha_n$ satisfies $\gamma(x)=\eta(x)$, then
$\sum_{j=1}^nr_j=1$. Hence $K=\conv_\setR(\alpha_1,\ldots,\alpha_n)$, and
thus~$K$ is compact. If a candidate for~$b_1$ is given, after passing to a
positive rational multiple we may assume that this $b_1$ lies in~$K$.
Otherwise we choose any non-zero element $b_1\in V^*\cap F\cap K$. If
$\eta\in [0,\infty)b_1$, we can use Step~\ref{step6} to extend $b_1$
(using generators of some extreme rays) to a $\setQ$-basis with the desired
properties. Otherwise we find $t>1$ such that $d:=b_1+t(\eta-b_1)$ lies in the
boundary of $F$ relative $\Span_\setR(F)$ (using that~$K$ is compact). Then
$d\in\algint(F')$ for some face $F'$ of~$F$ of dimension $0<m<\ell$.
By induction we find a $\setQ$-basis $b_2,\ldots,b_{m+1}$ of
$\Span_\setQ(F'\cap V^*)$ in $F'\cap V^*$, such that
$d\in\cone_\setR(b_2,\ldots, b_{m+1})$. Note that $b_1\not\in F'$ (otherwise
the convex combination $\eta$ of $b_1$ and $d$ would lie in the proper
face~$F'$ of~$F$, contradicting the assumption that $\eta\in\algint(F)$).
Hence $b_1\not\in F'-F'=\Span_\setR(F')$ (using that $F'$ is a face). Thus
$b_1,b_2,\ldots,b_{m+1}$ are $\setQ$-linearly independent and can be extended
to a $\setQ$-basis $b_1,\ldots, b_\ell\in F\cap V^*$ of
$\Span_\setQ(F\cap V^*)$, using Step~\ref{step6}.
It remains to observe that
$\eta=\frac{1}{t}d+(1-\frac{1}{t})b_1\in\cone_\setR(b_1,\ldots,b_{m+1})
\subseteq \cone_\setR(b_1,\ldots,b_\ell)$.
\end{step}

\begin{step} \label{step8}
If $\zeta=0$, we choose a $\setQ$-basis
$\beta^*_1,\ldots,\beta^*_k\in\Gamma^\star$ of $V^*$ such that
\begin{equation} \label{eq13}
\theta\in\setQ^+_0\beta^*_1+\cdots+\setQ^+_0\beta^*_k,
\end{equation}
which is trivial for $\theta=0$, and can be achieved by taking $\beta^*_1$ as
a positive rational multiple of $\theta$ otherwise. If $\zeta\neq 0$, let
$F$ be the minimal face of $\Gamma^\star_\setR$ containing $\zeta$. Then
$\zeta$ is in the interior of $F$ relative $\Span_\setR(F)$, and $F\cap V^*$
is dense in $F$. If $\theta\in F$, we let $\beta^*_1\in F\cap V^*$ be a
non-zero vector such that $\theta$ is a non-negative rational multiple of
$\beta^*_1$. By Step~\ref{step7}, we can extend $\beta^*_1$ to a
$\setQ$-basis $\beta^*_1,\ldots,\beta^*_\ell\in F\cap V^*$ of
$\Span_\setQ(F\cap V^*)$ such that
\begin{equation} \label{eq14}
\zeta\in\cone_\setR{(\beta^*_1,\ldots,\beta^*_\ell)},
\end{equation}
which in turn we extend to a $\setQ$-basis
$\beta^*_1,\ldots,\beta^*_k\in\Gamma^\star$ of $V^*$. If $\theta\notin F$, we
first find a $\setQ$-basis $\beta^*_1,\ldots,\beta^*_\ell\in F\cap V^*$ of
$\Span_\setQ{(F\cap V^*)}$ such that \eqref{eq14} holds (using
Step~\ref{step7}), and then extend
it to a $\setQ$-basis $\beta^*_1,\ldots,\beta^*_k\in\Gamma^\star$ of $V^*$
such that $\beta^*_{\ell+1}$ is a positive rational multiple of~$\theta$.
This is possible since $\theta\not\in F-F =\Span_\setR(F)$ (as
$\theta\in\Gamma^\star_\setR\setminus F$ and~$F$ is a face).
In either case, $\zeta\in\setR^+_0\beta^*_1+\cdots+\setR^+_0\beta^*_k$, and
\eqref{eq13} holds. Let $\beta_1,\ldots,\beta_k\in V$ be the basis dual to
$\beta^*_1,\ldots,\beta^*_k$, and write $\Beta:=\{\beta_1,\ldots,\beta_k\}$.
Then $\{\beta^*_1,\ldots,\beta^*_k\}\subseteq \Gamma^\star$ entails that
\begin{equation*}
 \setQ^+_0\beta_1+\cdots+\setQ^+_0\beta_k =
 \{\beta^*_1,\ldots,\beta^*_k\}^\star\supseteq\Gamma^{\star\star}
 \supseteq \Gamma.
\end{equation*}
After replacing each $\beta_\kappa$ by a positive rational multiple, we may
assume that $[\Beta]\supseteq\Gamma$.
\end{step}

\begin{step} \label{step9}
Using Kronecker's $\delta$, we define
\begin{equation*}
 \phi\colon[\Beta]\to[0,\infty),\quad
 \phi(\xi):=e^{-\zeta(\xi)}\cdot\delta_{0,\theta(\xi)}.
\end{equation*}
Then $\phi$ is a homomorphism, since $\theta(\Beta)\subseteq[0,\infty)$ by
\eqref{eq13} and $\delta_{0,\CDOT}\colon([0,\infty),+)\to([0,\infty),\cdot)$
is a homomorphism of monoids. If $\xi\in\Gamma'$, then
$\phi(\xi)=e^{-\rho(\xi)}=\psi(\xi)$ by \eqref{eq11} and \eqref{eq12}. If
$\xi\in\Gamma_0$, then $\theta(\xi)>0$ by \eqref{eq11} and thus
$\phi(\xi)=0=\psi(\xi)$. Finally, $\phi$ is bounded, since
$\{\zeta\}\subseteq\setR^+_0\beta^*_1+\cdots+\setR^+_0\beta^*_k$ and thus
$[\Beta]\subseteq\{\beta^*_1,\ldots,\beta^*_k\}^\star
\subseteq\{\zeta\}^\star_\setR$.
\end{step}
This completes the proof of the Density Lemma.
\end{proof}

%%%%%%%%%%%%%%%%%%%%%%%%%%%%%%%%%%%%%%%%%%%%%%%%%%%%%%%%%%%%%%%%%%%%%%%%%%%%%

\section{A Hahn-Banach separation theorem for rational polytopes} \label{S7}

%%%%%%%%%%%%%%%%%%%%%%%%%%%%%%%%%%%%%%%%%%%%%%%%%%%%%%%%%%%%%%%%%%%%%%%%%%%%%

\noindent
The possibility of separation theorems for polytopes in vector spaces over
ordered fields is already mentioned in \cite[p.\,287]{GK1998} (without
proof). The proof of the Extension Lemma, Step 4 of Section \ref{S6},
required the following

\begin{slemma}
Let $V$ be a finite-dimensional $\setQ$-vector space and
\mbox{$E=\{x_1,\ldots,x_m\}\subseteq V$} be a non-empty finite subset such
that $0\notin C:=\conv_\setQ(E)$. Then there exists a $\setQ$-linear
functional $\rho\colon V\to\setQ$ such that $\rho(E)\subseteq (0,\infty)$.
\end{slemma}

\begin{proof}
Let $W\subseteq V$ be the affine subspace generated by $E$. If $0\notin W$,
then there exists $\rho\in V^\star$ such that $\rho|_W=1$ and hence
$\rho|_E=1$. Now assume that $0\in W$. After replacing $V$ with $W$, we may
assume that $V=\aff_\setQ(E)$. We may also assume $V=\setQ^n$ for some $n$.
Then $\aff_\setR(E)=\setR^n$ in $\setR^n$, whence $C_\setR:=\conv_\setR(E)$
has non-empty interior in $V_\setR:=\setR^n$.

We claim that $0\notin C_\setR$.

If this is true, then there is $y\in\setR^n$ such that
$\langle{y,C_\setR}\rangle\subseteq(0,\infty)$ by the Hahn-Banach Separation
Theorem. Then $\langle{y,x_j}\rangle>0$ for $j=1,\ldots,m$. By continuity, we
find $w\in\setQ^n$ close to $y$ such that $\langle{w,x_j}\rangle>0$ for
$j=1,\ldots,m$. Then $\rho:=\langle{w,\CDOT}\rangle\in(\setQ^n)^\star$ is as
desired.

We now prove the claim by induction on $\dim_\setQ(V)$. If $\dim_\setQ(V)=1$,
then $E$ is a finite subset of $\setQ$ and $C=[x_\star,x^\star]\cap\setQ$,
where $x_\star$ and $x^\star$ is the minimum and maximum of $E$,
respectively. Since $0\notin C$, we deduce that
$\{x_\star,x^\star\}\subseteq(0,\infty)$ or
$\{x_\star,x^\star\}\subseteq(-\infty,0)$, entailing that also
$0\notin C_\setR=[x_\star,x^\star]$. Induction step: If $0\in\pa C_\setR$ is
in the boundary, then $0$ is contained in a face $\Phi\neq C_\setR$ of the
polytope~$C_\setR$ (see \cite[Theorem 5.6]{Br1983}). We have
$\Phi=\conv_\setR(E')$ with $E':=\Phi\cap E$,
by \cite[Theorems 7.2 and 7.3]{Br1983}. Then $\aff_\setR(E')$ is a proper
vector subspace of $V_\setR$ (see \cite[Corollary 5.5]{Br1983}), and hence
$\aff_\setQ(E')$ is a proper vector subspace of $V$. By induction,
$0\notin\conv_\setR(E')=\Phi$. We have reached a contradiction.

It remains to discuss the case where $0$ is in the interior of $C_\setR$. For
some $\eps>0$ we then have $w\in C_\setR$ for all
$w=(w_1,\ldots,w_n)\in\{-\eps,\eps\}^n$. Since $C$ is dense in $C_\setR$,
for each $w$ (as before) we find an element $u\in C$ such that
$\|w-u\|_\infty<\frac{\eps}{2}$\,. For each $j\in\{1,\ldots,n\}$, the $j$th
component $u_j$ of $u$ is then non-zero and has the same sign as $w_j$. Let
$U$ be the set of all $u$ as before, for $w$ ranging through
$\{-\eps,\eps\}^n$. Now the next lemma shows that
$0\in\conv_\setQ(U)\subseteq C$, which is a contradiction.
\end{proof}

Here, we used

\begin{lemma} \label{L2}
Let $U\subseteq\setQ^n$ be such that for all signs
$\sigma_1,\ldots,\sigma_n\in\{-1,1\}$, there exists
$u=(u_1,\ldots,u_n)\in U$ with $\sgn{u_j}=\sigma_j$ for each
$j\in\{1,\ldots,n\}$. Then $0\in\conv_\setQ(U)$.
\end{lemma}

\begin{proof}
By induction on $n$. The case $n=1$ is trivial. Given $U\subseteq\setQ^n$ and
signs $\sigma_1,\ldots,\sigma_{n-1}$, we find $u\in U$ with signs
$\sigma_1,\ldots,\sigma_{n-1},1$ and $v\in U$ with signs
$\sigma_1,\ldots,\sigma_{n-1},-1$. By the case $n=1$, there is
$w=(w_1,\ldots,w_n)\in\conv_\setQ{\{u,v\}}$ such that $w_n=0$. Thus
$w=(w_1,\ldots,w_{n-1},0)$, where $w_1,\ldots,w_{n-1}$ have signs
$\sigma_1,\ldots,\sigma_{n-1}$. Consider the first $n-1$ coordinates. Then
the induction hypothesis yields $0\in\conv_\setQ(U)$.
\end{proof}

%%%%%%%%%%%%%%%%%%%%%%%%%%%%%%%%%%%%%%%%%%%%%%%%%%%%%%%%%%%%%%%%%%%%%%%%%%%%%

\section{Extensions and applications} \label{S8}

%%%%%%%%%%%%%%%%%%%%%%%%%%%%%%%%%%%%%%%%%%%%%%%%%%%%%%%%%%%%%%%%%%%%%%%%%%%%%

\noindent
The Lévy type extension of Theorem \ref{T1} (cf. Lévy \cite{Le1934} for
Fourier series) is obtained by applying the Composition Lemma.

\begin{theorem} \label{T3}
Let $\Lambda\subseteq[0,\infty)$ be an additive semigroup with
$0\in\Lambda$, let $f$ be a holomorphic function defined on a region
$G\subseteq\setC$, and let $w\in\scrw$. Suppose that $a\in\scra_w$ satisfies
$\ol{\wt{a}(\setH)}\subseteq G$. Then there exists a function $c\in\scra_w$
such that $f\circ\wt{a}=\wt{c}$.
\end{theorem}

\begin{proof}
For $a\in\scra_w$ we have
\begin{equation*}
 \{h(a):h\in\Delta(\scra_w)\} \subseteq \ol{\{h_s(a):s\in\ol{\setH}\,\}} =
 \ol{\wt{a}(\setH)}\subseteq G,
\end{equation*}
and the assertion follows from Theorem \ref{T1} and the Composition Lemma.
\end{proof}

In particular, Theorems \ref{T1} and \ref{T3} apply to $\Lambda=\setN_0$. We
write $z=e^{-s}$ (a transformation which maps $\setH$ onto
$\setU\setminus\{0\}$) and associate with $a\in\scra_w(\setN_0)$ the power
series
\begin{equation*}
\wt{a}(z)=\wt{a}(e^{-s})=\sum_{n=0}^\infty\,a(n)\,z^n.
\end{equation*}
Since $\wt{a}$ is continuous on the compact set $\ol{\setU}$, we have
$\ol{\wt{a}(\setU)}=\wt{a}(\ol{\setU})$.

Then the weighted version of Wiener's inversion theorem for power series
reads as follows.

\begin{corollary} \label{C1}
For $w\in\scrw(\setN_0)$ the multiplicative group of the Banach algebra
$\scra_w(\setN_0)$ is
\begin{equation*}
 \scra^*_w(\setN_0) =
 \big\{a\in\scra_w(\setN_0):0\notin\wt{a}(\ol{\setU})\big\}.
\end{equation*}
\end{corollary}

The lemmas needed for the proof of Theorems \ref{T1} and \ref{T3} easily
extend to additive semigroups $\Lambda$ of \emph{product type},%
\footnote{In fact, every bounded character~$\psi$ of~$\Lambda$ is of the form $\psi(\lambda_1,\ldots,\lambda_r)=\prod_{\vr=1}^r\psi_\vr(\lambda_\vr)$
with bounded characters $\psi_\vr$ of $\Lambda_\vr$, which can be
approximated by $e^{-\lambda_\vr s_\vr}$ on a finite set
$\Gamma_\vr\subseteq\Lambda_\vr$. Thus $\psi(\lambda)$ can be approximated by
$e^{-\lambda\cdot s}$ on $\Gamma_1\times\cdots\times\Gamma_r$.}
by which we understand semigroups
$\Lambda=\Lambda_1\times\cdots\times\Lambda_r$ with additive semigroups
$\Lambda_\vr\subseteq[0,\infty)$ and $0\in\Lambda_\vr$ for $\vr=1,\ldots,r$.
Then the multidimensional versions of Theorems \ref{T1} and \ref{T3} cover
arithmetic functions of $r$ variables. We put
$\lambda\cdot s=\lambda_1s_1+\cdots+\lambda_rs_r$ for
$\lambda=(\lambda_1,\dots,\lambda_r)\in\Lambda$ and
$s=(s_1,\dots,s_r)\in\setC^{\,r}$. The $r$-dimensional Dirichlet series
associated with $a\colon\Lambda\to\setC$ is
\begin{equation} \label{eq15}
 \wt{a}(s) = \sum_{\lambda\in\Lambda}\,a(\lambda)\,e^{-\lambda\cdot s}
 \qquad(s\in\setC^r).
\end{equation}

Repeating the reasoning leading to Theorem \ref{T1}, we see that the
assertion of Theorem \ref{T1} remains true for additive semigroups
$\Lambda\subseteq[0,\infty)^r$ of product type. Indeed, we can dispense with
$\Lambda$ to be of product type.

\begin{theorem} \label{T4}
Let $L\subseteq[0,\infty)^r$ be an additive semigroup with $0\in L$ and
$w\in\scrw(L)$. If $a\in\scra_w(L)$ satisfies $0\notin\ol{\wt{a}(\setH^r)}$,
then $a$ is invertible in $\scra_w(L)$.
\end{theorem}

\begin{proof}
According to the Representation Lemma
$\scra^*_w(L)=\scra^*_1(L)\cap\scra_w(L)$ holds for any additive semigroup
$L$ with $0\in L$ and any admissible weight function $w$. Therefore it
suffices to show that $a\in\scra^*_1(L)$.

Denote by $\Lambda_\vr$ the set of $\vr$th components of $L$. Then
$\Lambda_\vr\subseteq[0,\infty)$ is an additive semigroup with
$0\in\Lambda_\vr$, $L$ is a subsemigroup of
$\Lambda=\Lambda_1\times\cdots\times\Lambda_r$, and
$\scra_1(L)\subseteq\scra_1(\Lambda)$. By the previous remark we have
$a\in\scra^*_1(\Lambda)$ and, in particular, $a(0)\neq 0$.

To obtain $a\in\scra^*_1(L)$, we have to show that $a^{-1}(\lambda)=0$ for
all $\lambda\in\Lambda\setminus L$. As a tool, let us consider the weight
functions $w_\rho$ on $L$ (and $\Lambda$) for $\rho\geq 0$ defined via
$w_\rho(\lambda):=e^{-\rho(\lambda_1+\cdots +\lambda_r)}$ (which are not
admissible if $\rho>0$, but define the weighted algebra
$\scra_{w_\rho}$). Choose $\rho>0$ so large that
\begin{equation}\label{eq20}
 \frac{1}{|a(0)|}\,\sum_{\lambda\in L\setminus\{0\}}|
 a(\lambda)|e^{-\rho(\lambda_1+\cdots+\lambda_r)} < 1.
\end{equation}
Then $a\in \scra_{w_\rho}^*(L)$, since
$a=a(0)\big(\ve+\frac{1}{a(0)}(a-a(0)\ve)\big)$, where
$\big\|\frac{1}{a(0)}(a-a(0)\ve)\big\|_{w_\rho}<1$ by~\eqref{eq20}. 
By Theorem \ref{T2}\,a) the inverse $b$ of $a$ in $\scra_{w_\rho}^*(L)$
exists. Then both $a^{-1}\in \scra_1(\Lambda)$ and $b$ are inverses of $a$
in the algebra $\scra_{w_\rho}(\Lambda)$ (which contains both
$\scra_1(\Lambda)$ and $\scra_{w_\rho}(L)$ as unital subalgebras), and hence
coincide. Thus $a^{-1}=b\in\scra_1(\Lambda)\cap\scra_{w_\rho}(L)=\scra_1(L)$.
\end{proof}

Similarly we obtain a multi-dimensional weighted Wiener-Lévy type theorem
by using the Composition Lemma and the Density Lemma in the version
$\sigma(a)\subseteq\ol{\wt{a}(\setH^r)}$ for $a\in\scra_w(L)$ for additive
semigroups $L\subseteq[0,\infty)^r$ with $0\in L$.

\begin{theorem} \label{T5}
Let $L\subseteq[0,\infty)^r$ be an additive semigroup with $0\in L$, $f$ a
holomorphic function defined on a region $G\subseteq\setC$, and
$w\in\scrw(L)$. If $a\in\scra_w(L)$ satisfies
$\ol{\wt{a}(\setH^r)}\subseteq G$, then there exists a function
$c\in\scra_w(L)$ such that $f\circ\wt{a}=\wt{c}$.
\end{theorem}

Let $\Lambda\subset[0,\infty)$ be a free additive semigroup with
$0\in\Lambda$ and finite or countable generating set~$\Beta$. Arithmetical
applications are usually based on the associated free \emph{multiplicative}
semigroup $\caln:=e^\Lambda\in[1,\infty)$ with $1\in\caln$ and the finite
or countable generating set $\calp:=e^\Beta$ of \emph{prime elements}. By
definition each $n\in\caln$ has a unique factorization of the form
$n=p_1^{\nu_1}\cdots p_r^{\nu_r}$ with distinct $p_1,\ldots,p_r\in\calp$
and positive integer exponents $\nu_1,\ldots,\nu_r$, apart of the order of
prime element powers. As usual, the empty product has the value $1$, and
elements $m,n\in\caln$ not having any common prime divisors are called
\emph{coprime}. The prototype of such multiplicative semigroups is $\setN$
generated by the set $\setP$ of primes.

The functions $a\colon\caln\to\setC$ now correspond to the Dirichlet series
\begin{equation*}
\wt{a}(s):=\sum_{n\in\caln}\,\frac{a(n)}{n^s} \qquad (s\in\setC),
\end{equation*}
and the functions $\om\colon\caln\to[1,\infty)$ with $\om(n)\geq\om(1)=1$
that are submultiplicative, i.e. $\om(mn)\leq \om(n)\,\om(m)$ for all
$m,n\in\caln$, and satisfy
$\lim_{k\to\infty}{\!\textstyle\sqrt[k]{\om(n^k)}}=1~$ for all $n\in\caln$
form the class $\scrw_1:=\scrw_1(\caln)$ of admissible weight functions.
By Theorem \ref{T1} each $\om\in\scrw_1$ yields a Banach algebra
$\scra_\om:=\scra_\om(\caln)$ of functions $a\colon\caln\to\setC$, endowed
with the linear operations and convolution
\begin{equation} \label{eq21}
 (a*b)(n):=\sum_{\substack{\ell,m\in\caln \\ \ell\,m=n}} a(\ell)\,b(m)
 \qquad(n\in\caln)
\end{equation}
and with bounded $\om$-norm $\|a\|_\om:=\sum_{n\in\caln} |a(n)|\,\om(n)$.
The unity in $\scra_\om$ is $\ve:=\delta_1$, and $\scra_\om$ has the
multiplicative group
\begin{equation*}
\scra^*_\om(\caln)=\{a\in\scra_\om:0\notin\ol{\wt{a}(\setH)}\}.
\end{equation*}

We consider the class $\scrm$ of \emph{multiplicative functions}
$a\colon\caln\to\setC$, i.e. $a(mn)=a(m)\,a(n)$ for all coprime $m,n\in\caln$
and $a(1)=1$. Since $\scrm$ is closed under convolution and inversion,
$\scrm$ is a group. For each $\om\in\scrw_1$, the Dirichlet series
$\wt{a}(s)$ of $a\in\scrm_\om:=\scrm\cap\scra_\om(\caln)$ converges
absolutely for $s\in\ol{\setH}$ and has the Euler product representation
\begin{equation*}
\wt{a}(s) = \prod_{p\in\calp}\,
 \Big(1+\frac{a(p)}{p^s}+\frac{a(p^2)}{p^{2s}}+\cdots\Big) =:
 \prod_{p\in\calp}\,\wt{a}_p(s).
\end{equation*}
Since each factor $\wt{a}_p(s)$ represents an absolutely convergent power
series in $z=p^{-s}\in\ol{\setU}$ and thus a continuous function on the
compact disc $\ol{\setU}$, its weighted inversion according to Corollary
\ref{C1} with $w(k)=\om(p^k)$ for $k\in\setN_0$ only requires
$0\notin\wt{a}_p(\ol{\setH})$, instead of $0\notin\ol{\wt{a}_p(\setH)}$.

\begin{corollary} \label{C3}
For $\om\in\scrw_1$ the class $\scrm_\om$ is a unital subsemigroup of $\scrm$
under the convolution \eqref{eq21} with the multiplicative group
\begin{equation} \label{eq22}
 \scrm^*_\om=\{a\in\scrm_\om:
 \wt{a}_p(s)\neq 0 \text{ for all }s\in\ol{\setH} \text{ and }p\in\calp\}.
\end{equation}
\end{corollary}

For infinite generating set $\calp$, there are many arithmetically
interesting multiplicative functions $a$ that do not belong to $\scrm_\om$,
particularly those for which the series $\sum_{p\in\calp} a(p)\,\om(p)$ does
not converge absolutely. Therefore we extend $\scrm_\om$ by partly replacing
the $\om$-norm with the mean square $\om$-norm (for $\caln=\setN$ cf.~Lucht
\cite{Lu1993} and, with $\om=1$, Heppner and Schwarz \cite{HS1983}):

\begin{proposition} \label{P3}
Let $\caln\subset[1,\infty)$ be a free multiplicative semigroup with
$1\in\caln$ and countable generating set $\calp$. Then, for $\om\in\scrw_1$,
the class
\begin{equation*}
 \scrg_\om=\Big\{a\in\scrm:\sum_{p\in\calp}\,|a(p)|^2\,\om^2(p)<\infty
 \text{ and }
 \sum_{\substack{p\in\calp \\ k\geq 2}}\,|a(p^k)|\,\om(p^k)<\infty\Big\}
\end{equation*}
is a unital subsemigroup of $\scrm_\om$ under the convolution \eqref{eq21}
with the multiplicative group
\begin{equation*}
 \scrg^*_\om=\big\{a\in\scrg_\om:\wt{a}_p(s)\neq 0\,\text{ for all }\,
 s\in\ol{\setH}\text{ and }p\in\calp\big\}.
\end{equation*}
\end{proposition}

Note that $\scrg_\om=\scrm_\om$ for finite $\calp$, whereas
$\scrm_\om\varsubsetneq\scrg_\om$ for infinite $\calp$. Further, for
$a\in\scrg_\om$, the series $\sum_{p\in\calp} a(p)$ even might diverge.

\begin{proof}[Proof of Proposition \ref{P3}.]
The submultiplicativity of the $\om$-norm combined with the Cauchy-Schwarz
inequality entails that $\scrg_\om$ is closed under $*$\,, and trivially
$\ve\in\scrg_\om$. For $a,b\in\scrg^*_\om$ and $p\in\calp$ we have
$a_p\,,b_p\in\scrg^*_\om$ and $(a*b)_p\wt{\phantom{\iota}}(s)=
\wt{a}_p(s)\,\wt{b}_p(s)\neq 0$ for $s\in\ol{\setH}$. Hence $\scrg^*_\om$ is
also closed under $*$\,. It remains to verify that $a\in\scrg^*_\om$ implies
$a^{-1}\in\scrg_\om$.

From Corollary \ref{C1} applied to $\wt{a}_p(s)$ with $w(k)=\om(p^k)$ and
$z=p^{-s}$ for $p\in\calp$ fixed and $k\in\setN_0$ we infer that
$a_p^{-1}\in\scrg_\om$ for each $p\in\calp$. To transfer this invertibility
property from all factors $a_p$ to $a$ we consider the Euler product
representation of $\wt{a}(s)$ written as
\begin{equation*}
 \wt{a}(s)=\prod_{p\leq p_0}\,\wt{a}_p(s)\cdot
 \prod_{p>p_0}\,\Big(1-\frac{a(p)}{p^s}\Big)^{-1}\cdot
 \prod_{p>p_0}\,\Big(1-\frac{a(p)}{p^s}\Big)\,\wt{a}_p(s).
\end{equation*}
It corresponds to the decomposition
\begin{equation} \label{eqII.4.11'}
 a=\Big(\underset{p\leq p_0}{\AST}\,a_p\Big)*b*h
\end{equation}
with $p_0\in\calp$ suitably large, and $b,h\in\scrm$ defined by
\begin{align*}
 b(p^k)
 & = \begin{cases}
     a^k(p) & \text{ for }\,p>p_0,\,k\in\setN_0, \\
     0      & \text{ otherwise,}
     \end{cases}\\[.6ex]
 h(p^k)
 & = \begin{cases}
     a(p^k)-a(p^{k-1})\,a(p) & \text{ for }\,p>p_0,\,k\in\setN, \\
     0                       & \text{ otherwise.}
     \end{cases}
\end{align*}
Obviously $h(p)=0$ for $p\in\calp$ and $b(p)=a(p)$ for all $p>p_0$. We
choose $p_0\in\calp$ sufficiently large such that both estimates
\begin{equation*}
|a(p)|\,\om(p)\leq\frac{1}{2} \text{ for }p>p_0\quad\text{ and }\quad
 \sum_{\substack{p>p_0 \\ k\geq 2}}\,|h(p^k)|\,\om(p^k)\leq\frac{1}{2}
\end{equation*}
hold. Then $b\in\scrg_\om^*$ and $b^{-1}=\mu b\in\scrg_\om$, as $b$ is
completely multiplicative and $|\mu(n)|\leq 1$ for $n\in\setN$. Further,
the submultiplicativity of $\om$ together with the Cauchy-Schwarz
inequality yields $h\in\scrg_\om$.

In order to verify $h\in\scrg^*_\om$ we conclude from $h^{-1}*h=\ve$ that
$h^{-1}(p)=h(p)=0$ for all $p\in\calp$, $h(p^k)=0$ for all $p\leq p_0$
and $k\in\setN$, and
\begin{equation*}
 h^{-1}(p^k)=-\,\sum_{0\leq j\leq k-2}\,h^{-1}(p^j)\,h(p^{k-j})
 \qquad(p>p_0,\,k\geq 2).
\end{equation*}
From this combined with the submultiplicativity of $\om$ we obtain
\begin{align*}
\Sigma:
 & = \sum_{\substack{p^k\leq x \\ k\geq 2}}\,|h^{-1}(p^k)|\,\om(p^k) \\
 & \leq \sum_{\substack{p^k\leq x \\ k\geq 2}}\,
        \sum_{0\leq j\leq k-2}\,|h^{-1}(p^j)|\,w(p^j)\cdot
        |h(p^{k-j})|\,\om(p^{k-j}) \\
 & =    \sum_{\substack{p^k\leq x \\ k\geq 2}}\,|h(p^k)|\,\om(p^k)+
        \sum_{\substack{p^{j+\ell}\leq x \\ j,\ell\geq 2}}\,
        |h^{-1}(p^j)|\,\om(p^j)\cdot|h(p^\ell)|\,\om(p^\ell) \\
 & \leq \big(1+\Sigma\big)\,\sum_{\substack{p^\ell\leq x \\ \ell\geq 2}}\,
        |h(p^\ell)|\,\om(p^\ell)\leq\frac{1}{2}\,\big(1+\Sigma\big).
\end{align*}
Hence $\Sigma\leq 1$ and $h^{-1}\in\scrg_\om$. Now \eqref{eqII.4.11'} entails
that $a\in\scrg_\om$ is a convolution of finitely many elements of
$\scrg^*_\om$ and thus $a^{-1}\in\scrg_\om$.
\end{proof}

An important arithmetical application of weight functions is based on
Proposition \ref{P3} and the notion of related arithmetical functions of
$\scra:=\scra(\caln)$. For $\om\in\scrw_1$, we say that $a\in\scra$ is
\emph{$\om$-related} to $b\in\scra^*$, if $h:=a*b^{-1}\in\scra_\om$. Via
Proposition \ref{P3} the $\om$-relationship of functions $a\in\scrg_\om$,
$b\in\scrg^*_\om$ from their values at prime elements can easily be
verified. This leads, for instance, to a new concept demanded for by
Knopfmacher \cite[Ch.\,7,\,\S\,5]{Kn1975} of Ramanujan expansions of
arithmetic functions (cf. Lucht \cite{Lu2010}).

%%%%%%%%%%%%%%%%%%%%%%%%%%%%%%%%%%%%%%%%%%%%%%%%%%%%%%%%%%%%%%%%%%%%%%%%%%%%%

%%%%%%%%%%%%%%%%%%%%%%%%%%%%%%%%%%%%%%%%%%%%%%%%%%%%%%%%%%%%%%%%%%%%%%%%%%%%%

\end{document}